\documentclass[11pt]{article}

\usepackage{amssymb}
\usepackage{amsmath}
\usepackage{mathrsfs}
\usepackage{amsfonts}
\usepackage{vmargin}
\usepackage{amsthm}
\usepackage{graphicx}
\usepackage{authblk}
\usepackage{xfrac}
\usepackage[T1]{fontenc}

\long\def\symbolfootnote[#1]#2{\begingroup%
\def\thefootnote{\fnsymbol{footnote}}\footnote[#1]{#2}\endgroup}

\setmarginsrb{20mm}{20mm}{20mm}{20mm}{10mm}{10mm}{10mm}{10mm}

\newtheoremstyle{remark}
  {}{}{}{}{\bfseries}{.}{.5em}{{\thmname{#1 }}{\thmnumber{#2}}{\thmnote{ (#3)}}}

\RequirePackage{amsthm}

\theoremstyle{definition}
\newtheorem{defi}{Definition}[section]
\newtheorem{rem}[defi]{Remark}
\theoremstyle{plain}

\newtheorem{tw}[defi]{Theorem}
\newtheorem{lem}[defi]{Lemma}
\newtheorem{cor}[defi]{Corollary}
\newtheorem{prop}[defi]{Proposition}

\def\vint{\mathop{\mathchoice%
          {\setbox0\hbox{$\displaystyle\intop$}\kern 0.22\wd0%
           \vcenter{\hrule width 0.6\wd0}\kern -0.82\wd0}%
          {\setbox0\hbox{$\textstyle\intop$}\kern 0.2\wd0%
           \vcenter{\hrule width 0.6\wd0}\kern -0.8\wd0}%
          {\setbox0\hbox{$\scriptstyle\intop$}\kern 0.2\wd0%
           \vcenter{\hrule width 0.6\wd0}\kern -0.8\wd0}%
          {\setbox0\hbox{$\scriptscriptstyle\intop$}\kern 0.2\wd0%
           \vcenter{\hrule width 0.6\wd0}\kern -0.8\wd0}}%
          \mathopen{}\int}
\def\={\hspace{-3mm}&=&\hspace{-3mm}}

\renewcommand{\r}{\mathbb{R}}
\newcommand{\n}{\mathbb{N}}

\newcommand{\ddt}{\frac{{d}}{dt}}
\newcommand{\iab}{\int_a^b}
\let \epsilon \varepsilon
\let \phi \varphi
\newcommand{\lab}[1]{L^{#1}(a,b)}
\newcommand{\hab}{H^1(a,b)}
\newcommand{\habl}{H^2(a,b)}
\newcommand{\habo}{H^1_0(a,b)}

\newcommand{\ul}{\sfrac{1}{2}}
\newcommand{\loc}{\operatorname*{loc}}
\newcommand{\T}{\mathcal{T}}
\newcommand{\M}{\mathcal{M}}

\begin{document}
\date{\today}
\title{\bf Time-asymptotics of a heated string}

\author[1]{Piotr Micha{\l} Bies\footnote{piotr.bies@pw.edu.pl}} \author[2,1]{Tomasz Cie\'slak\footnote{cieslak@impan.pl}}
\affil[1]{Faculty of Mathematics and Information Sciences, Warsaw University of Technology, Ul. Koszykowa 75, 00-662 Warsaw, Poland}
\affil[2]{Institute of Mathematics, Polish Academy of Sciences, 00-656 Warsaw, Poland}
\maketitle

\begin{abstract}
In the present paper, we study a model of a thermoelastic string that is initially heated. We classify all the possible asymptotic states
when time tends to infinity of such a model. Actually, we show that whatever the initial data is, a heated string must converge to a flat, steady 
string with uniformly distributed heat. The latter distribution is calculated from the energy conservation. In order to obtain the result, we need to 
take a few steps. In the first two steps, time-independent bounds from above and from below (by a positive constant) of the temperature are obtained. 
This is done via the Moser-like iteration. The lower bound is obtained via the Moser iteration on the negative part of the logarithm of temperature. In 
the third step, we obtain a time-independent higher-order estimate, which yields compactness of a sequence of the values of the solution when time 
tends to infinity. Here, an estimate involving the Fisher information of temperature, together with a recent functional inequality from \cite{CFHS} and 
an $L^2(L^2)$ estimate of the gradient of entropy, enable us to arrive at a tricky Gr\"{o}nwall type inequality. Finally, in the last steps, we define 
the dynamical system on a proper functional phase space and study its $\omega$-limit set. To this end, we use, in particular, the quantitative version 
of the second principle of thermodynamics. Also, the entropy dissipation term and the bound of the entropy from below are useful when identifying the 
structure of the $\omega$-limit set.
\end{abstract}

\bigskip

\noindent
{\bf Keywords}: 1D thermoelasticity, time-independent estimates, convergence to steady states, second principle of thermodynamics
\medskip

\noindent
\emph{Mathematics Subject Classification (2020): 35B35, 35B40, 35Q74}
\medskip

\section{Introduction}

This article is devoted to the study of the asymptotic behavior of a heated string. We consider a string with fixed ends.
The function $u:(0,T)\times (a,b)\rightarrow \r$ denotes the position of the string, and a positive function $\theta:(0,T)\times (a,b)\rightarrow \r$ 
denotes its temperature. The following system of equations is satisfied by a couple $(u,\theta)$.

\begin{equation}\label{eq}
\begin{cases}
u_{tt}-u_{xx}=\mu\theta_x,&\textrm{in }(0,\infty)\times (a,b),\\
\theta_t-\theta_{xx}=\mu\theta u_{tx},&\textrm{in }(0,\infty)\times (a,b),\\
u(\cdot,a)=u(\cdot,b)=\theta_x(\cdot,a)=\theta_x(\cdot,b)=0,&\textrm{on } (0,\infty),\\
u(0,\cdot)=u_0,\ u_t(0,\cdot)=v_0,\ \theta(0,\cdot)=\theta_0>0,
\end{cases}
\end{equation}
where $v_0$ is the initial velocity, $\mu$ is a material constant. The above system is a particular example of the $1d$ thermoelasticity model for the 
details of the physical derivation, we refer to \cite{TC2, slemrod}. Let us only mention here that \eqref{eq} corresponds to the choice of the 
Helmholtz free energy function
$\Psi(\theta,u_x):=\theta \log \theta -\theta +\mu\theta  u_x +\frac12  u_x^2$ in \cite{slemrod} (Fick's law for the spread of the 
temperature is also assumed).

The local-in-time existence and uniqueness of regular solutions to even more general $1d$ thermoelasticity systems have already been established in 
\cite{slemrod}. Similarly, the global existence of solutions starting in the neighborhood of the steady states. Further results concerning 
local-in-time solutions or solutions in the neighborhood of the steady states (including refinements like proving the positivity of the temperature) 
have been established in \cite{TC1,hrusa_tarabek, racke,racke_shibata, rsj}. The convergence of the solutions of the thermoviscoelastic system to the 
solutions to \eqref{eq}, when viscosity goes to $0$, has been established by \cite{daf_chen}. Finally, the weak solutions in the higher-dimensional 
case
(together with weak-strong uniqueness) are the purpose of \cite{TC2,galanopoulou}. The existence (and uniqueness) results for the general 
thermoelasticity problems and the state of the art in the linear approximation of the latter are presented in the notes of Racke (\cite{Racke_not}).

Dafermos and Hsiao have shown finite-time blowup of solutions to some thermoelasticity $1d$ problem with a nonlinear term $(p(u)u_x)_x$ instead of 
$u_{xx}$ in the upper equation of \eqref{eq}, see \cite{daf_hsiao}. A similar refined result (in terms of physical relevance) has been obtained, using 
similar methods, in \cite{hrusa_m}. Recently, a global existence of unique solutions in the case of \eqref{eq} has been obtained in \cite{BC}. There, 
we utilized a Fisher information-based functional to reach the result. Having defined the unique global solutions, a natural question of asymptotic 
behavior occurs. On the one hand, it has been known that solutions starting next to the steady states converge towards them, see, for instance, 
\cite{rsj} or \cite{Racke_not} for a wider picture. In the linear approximation, a $1d$ case was also well understood, see, for instance, \cite{Henry}. 
Here, we study the asymptotic behavior for general initial data. We show that whatever the data is, weak solutions constructed in \cite{BC} (we shall 
recall the definition in the next section) converge with time to the particular steady state, namely a flat string with uniformly distributed 
temperature.

Indeed, the following theorem is the main result of the present article.
\begin{tw}\label{main_tw_gr}
Let us assume $u$ and $\theta$ are a weak solution of \eqref{eq} starting from initial data $u_0$ and $\theta_0>0$, respectively. Then, the following 
convergences
\begin{align*}
u(t,\cdot)&\to 0\textrm{ in }\habo \textrm{ when }t\to\infty,\\
u_t(t,\cdot)&\to 0\textrm{ in }\lab2 \textrm{ when }t\to\infty,\\
\theta(t,\cdot)&\to \theta_{\infty}\textrm{ in }\lab2 \textrm{ when }t\to\infty,
\end{align*}
hold, where $\theta_{\infty}:=\left(\frac12\iab v_0^2dx+\frac12\iab u_{0,x}^2dx+\iab\theta_0 dx\right)/(b-a)$.
\end{tw}

Our theorem has a clear physical meaning. It shows that the damping effect of the heat propagation itself is so stabilizing that natural oscillations 
appearing in the string equation decay when coupled with heat. It is even more interesting in light of the recent result in \cite{Bies}, where it is 
pointed out that the given right-hand side in the heat equation generically cannot enforce the decay of oscillations. Let us discuss methods of the 
proof, emphasizing the role of physics-related estimates.

At first, we need to get the time-independent bounds of temperature. Both bounds are obtained using the Alikakos variation (\cite{alik}) on the Moser 
method, see \cite{Moser}. This is a delicate iterative method used in the case of upper-temperature bounds in thermoelasticity already in 
\cite{daf_hsiao}. A similar use was done in \cite{hrusa_m}. More interesting and novel is our iterative time-independent bound from below (by a 
positive constant) of the temperature. It is essential in the identification of the limits of the solutions when time goes to infinity. This is done 
via the Moser iteration on the negative part of the logarithm of temperature. Both Moser's iterations are presented in Section \ref{iteration}. Next, 
we obtain a time-independent higher-order estimate, which is necessary for the compactness estimates when time tends to infinity. Here, we use a 
functional from \cite{BC} involving the Fisher information of the temperature. However, since (unlike in \cite{BC}) we need a time-independent 
estimate, we need to treat the functional in a different way. On the one hand, we utilize a recent functional inequality from \cite{CFHS}. On the other 
hand, an $L^2(L^2)$ estimate of the gradient of entropy together with the $L^\infty$ estimate obtained via the Moser method, enable us to find a tricky 
Gr\"{o}nwall type inequality, which allows us to conclude a required compactness estimate. The latter is contained in Section \ref{compactness}. 
Finally, in the last section, we define the dynamical system on a properly chosen functional phase space. This can be done thanks to the uniqueness of 
solutions obtained in \cite{BC}. Then, the $\omega$-limit set of such a dynamical system is found for every initial data. For each initial data it 
turns out to contain only one point, a flat string with the uniformly distributed heat that can be identified knowing the initial data. Our method is a 
delicate reasoning making use of the quantitative version of the second principle of thermodynamics, including the entropy dissipation. The lower bound 
of the entropy is also very useful when identifying the asymptotic state.

\section{Preliminaries}%\label{preliminaries}
This section is devoted to recalling important information concerning the properties of solutions to \eqref{eq}, which we need in the sequel. Most of 
it is known, we state it in details for further reference.

First of all, on multiplying the upper equation in \eqref{eq} by $u_t$ and then integrating both equations of \eqref{eq}, one obtains the following 
energy identity.
\begin{prop}\label{energy_cons}
Regular solutions $\theta, u$ of \eqref{eq} satisfy
\begin{align}\label{balance}
\frac12\iab u_t^2dx+\frac12\iab u_x^2dx+\iab\theta dx=\frac12\iab v_0^2dx+\frac12\iab u_{0,x}^2dx+\iab\theta_0 dx.
\end{align}
\end{prop}
Next, we recall the essential identity in \cite{BC}, which led to the global existence result there.
\begin{lem}\label{estlem}
Let us assume that $u$ and $\theta$, $\theta>0$ are regular enough solutions of (\ref{eq}). Then, the following identity holds
\begin{align*}
\frac12\ddt\left(\int_a^b\frac{\theta_x^2}{\theta}dx+\int_a^bu_{tx}^2dx+\int_a^bu_{xx}^2dx\right)=-\int_a^b\theta\left[(\log\theta)_{xx}\right]^2dx+\frac{\mu}2\int_a^b\frac{\theta^2_x}{\theta}u_{tx}dx.
\end{align*}
\end{lem}

Next, let us recall the precise notion of solutions to \eqref{eq} constructed in \cite{BC}.
\begin{defi}\label{weakdef}
We say that $(u,\theta)$ is a solution to (\ref{eq}) if:
\begin{itemize}
\item The initial data is of regularity
\begin{align*}
u_0\in H^2(a,b)\cap\habo,\quad v_0\in\habo,\quad \theta_0\in\hab.
\end{align*}
We also require that there exists $\tilde{\theta}>0$ such that $\theta_0(x)\geq\tilde{\theta}$ for all $x\in[a,b]$.
\item Solutions $\theta$ and $u$ satisfy
\begin{align*}
u&\in L^{\infty}_{\loc}(0,\infty;H^2(a,b))\cap W^{1,\infty}_{\loc}(0,\infty;\habo)\cap W^{2,\infty}_{\loc}(0,\infty;\lab{2}),\nonumber\\
\theta&\in L^{\infty}_{\loc}(0,\infty;\hab)\cap H^1_{\loc}(0,\infty; \lab{2}).
\end{align*}
\item The momentum equation
\begin{align*}
 u_{tt}- u_{xx}=\mu\theta_x
\end{align*}
is satisfied for almost all $(t,x)\in (0,\infty)\times (a,b)$.
\item The entropy equation
\begin{align*}
\iab\theta_t\psi dx+\iab\theta_x\psi_x dx =\mu\iab\theta u_{tx}\psi dx
\end{align*}
holds for all $\psi\in C^{\infty}[a,b]$ and for almost all $t\in (0,\infty)$.
\item Initial conditions are attained in the following sense:
\[
\theta \in C([0,\infty);L^2(a,b)),
\]
\[
u\in C([0,\infty);\habo),\ u_t \in C([0,\infty);L^2(a,b)).
\]
\end{itemize}
\end{defi}

\begin{rem}\label{rem1}
Let us emphasize that squaring the lower equation in \eqref{eq} and next integrating, one increases the regularity of solutions, obtaining that 
$\theta\in L^2_{\loc}(0,\infty;H^2(a,b))$, $\theta_t\in L^2_{\loc}(0,\infty;L^2(a,b))$. This way, we see that both $\theta_t$ and $\theta_{xx}$ are 
defined a.e. in $(0,T)\times (a,b)$ and that actually the entropy equation is satisfied for almost all $(t,x)\in (0,\infty)\times(a,b)$. Moreover, by 
the Gelfand triple (see, for instance, \cite{Boyer}), we notice that $\theta\in C([0,\infty);\hab)$.
\end{rem}

In \cite{BC}, the following existence and uniqueness result for solutions of \eqref{eq} is proven.

\begin{tw}\label{main_theorem}
For any constant $\mu\in \r$, there exists a unique global-in-time solution of \eqref{eq} in the sense of Definition \ref{weakdef} with a positive 
temperature.
\end{tw}
The theorem is proven via the approximation with the half-Galerkin method. The estimates of an approximation sequence are strong enough to show that in 
the limit \eqref{eq} is satisfied. We shall need the approximation sequence again, so we recall its definition below. The existence of the solutions to 
the following approximating problem is given in \cite{TC2}. Let $\{\phi_i\}$ be a smooth basis of $\habo$ and let us denote 
$V_n=\operatorname{span}\{\phi_k\}_{1\leq k\leq n}$.

\begin{defi}\label{apdef}
We say that $u_n\in W^{1,\infty}_{\loc}(0,\infty;V_n)$ and $\theta_n\in L^2_{\loc}(0,\infty;H^1(a,b))$ are solutions of an approximate problem if for 
all $\phi\in V_n$ and $\psi\in H^1(a,b)$ the following equations are satisfied in $\mathcal{D}'(0,\infty)$
\begin{align*}
\frac{d^2}{dt^2}\iab u_n\phi dx+\iab u_{n,x}\phi_x dx&=-\mu\iab\theta_n\phi_x dx,\\
\ddt\iab\theta_n\psi dx+\iab \theta_{n,x}\psi_xdx&=\mu\iab\theta_n u_{n,tx}\psi dx.
\end{align*}
The initial identities
\begin{align*}
\theta_n(0)=\theta_{0},\quad u_n(0)=P_{V_n}u_0,\quad u_{n,t}(0)=P_{V_n}v_0
\end{align*}
hold, where $P_{V_n}$ is an orthogonal projection of the corresponding spaces onto $V_n$.
\end{defi}
\cite[Proposition 1 and Lemma 4.2]{TC2} states the following.
\begin{prop}\label{aproxex}
For every $n\in\mathbb{N}$, there exists a solution $(u_n,\theta_n)$ of the approximate problem in the sense of Definition \ref{apdef}. Moreover, one 
has $\theta_n\in L^2_{\loc}(0,\infty;H^2(a,b))\cap H^1_{\loc}(0,\infty;L^2(a,b))$ and $\theta_n(x,t)>0$ for almost all 
$(t,x)\in[0,\infty)\times[a,b]$.
\end{prop}
Due to the regularity of Galerkin bases, we have that $u_n\in C^{3,\infty}([0,T]\times[a,b])$ (see \cite{evans}). Similarly, by standard regularity 
theory of parabolic equations, we have $\theta_n\in C^{1,2}([0,T]\times[a,b])$ (see \cite[Chapter VII]{lady} or \cite[Chapter 5]{frid}).

Next, we introduce a notation $\tau:=\log\theta$.  The following observation related to $\tau$ will be crucial both in the compactness estimates as 
well as the asymptotic studies. The part \eqref{taul2} can be found already in \cite{TC2}. For the reader's convenience, we present the details.
\begin{prop}\label{ogrzpop}
Let $\theta$ be a weak solution of the problem \eqref{eq} in the sense of Definition \ref{apdef}. Then, the following equation
\begin{align*}
\ddt\iab\tau dx=\iab\tau_x^2dx
\end{align*}
is satisfied for almost all $t\in (0,\infty)$. Moreover, we have that
\begin{align}\label{taul2}
\tau_x\in L^2((0,\infty)\times (a,b)).
\end{align}
\end{prop}
\begin{proof}
By Remark \ref{rem1} the second equation in \eqref{eq} is satisfied pointwise. Hence, we obtain
\begin{align}\label{taurownanie}
\tau_t-\tau_{xx}-\tau^2_x=\mu u_{tx}.
\end{align}
We integrate it over $[a,b]$ to get
\begin{align*}
\frac{d}{dt}\iab\tau dx=\iab\tau_x^2dx.
\end{align*}
Next, we integrate over $[0,t]$ for an arbitrary $t>0$. We have
\begin{align*}
\iab\tau dx=\int_0^t\iab\tau_x^2dxds+\iab\tau_0dx
\end{align*}
where $\tau_0:=\log\theta_0$. Adding the latter to \eqref{balance}, we obtain
\begin{align*}
\frac12\iab u_t^2dx&+\frac12\iab u_x^2dx+\iab\theta dx-\iab\tau dx+\int_0^t\iab\tau_x^2dxds\\
&=\frac12\iab v_0^2dx+\frac12\iab u_{0,x}^2dx+\iab\theta_0 dx-\iab\tau_0 dx<\infty
\end{align*}
since
\[
\iab\theta dx-\iab\tau dx>0.
\]
The proof is finished.
\end{proof}
Notice that the first claim of Proposition \ref{ogrzpop} is a quantitative version of the second principle of thermodynamics.
It will be particularly useful in studying the asymptotic behavior of solutions. On the other hand, \eqref{taurownanie} turns out to be crucial
in obtaining the compactness estimates. Finally, notice that the proof of Proposition \ref{ogrzpop} and \eqref{balance} imply that there exists $C>0$ 
such that independently of time
\begin{align}\label{tauosz}
\iab |\tau| dx<C.
\end{align}
The same proposition applies to the approximation sequence $\tau_n:=\log\theta_n$ for all $n\in\n$.

Finally, we recall an identity from \cite{CFHS}.
\begin{lem}\label{cfhslem}
Let us consider a positive function $\psi\in C^2([a,b])$, $\psi_x(a)=\psi_x(b)=0$, then the following inequality
$$
\iab\left[\left(\psi^{\ul}\right)_{xx}\right]^2dx\leq \frac{13}{8}\iab\psi\left[\left(\log\psi\right)_{xx}\right]^2dx
$$
holds.
\end{lem}

\section{Time-independent estimates}\label{iteration}

The present section is devoted to the time-independent estimates of the temperature. On the one hand, the bound of the $L^\infty$ norm of $\theta$ is 
obtained via the Moser type (\cite{Moser}) iteration in the spirit of Alikakos (\cite{alik}). The latter is a delicate iterative procedure, but it is 
not fully new in the thermoelasticity. It has already been used in this context in \cite{daf_hsiao}. Still, it is a crucial part of our reasoning. In 
particular, in the following corollary, we shall utilize the $L^\infty$ bound to estimate the $L^2(L^2)$ norm of the gradient of $\theta^{\ul}$ 
independently on time. The content of the corollary will be essential in the compactness argument in Section \ref{compactness}. Last but not least, the 
last estimate of the present section is a time-independent bound from below of the temperature. Again, it is achieved via the Moser-type iteration. 
This time, however, the iteration is developed on the negative part of the entropy (the logarithm of temperature), requires the precise utilization of 
\eqref{taurownanie}, and seems a novel step in the field. The claim of the last theorem will be essential in identifying the asymptotic states of a 
heated string.
\begin{tw}\label{oszth}
Let $\theta$ be a weak solution of the problem \eqref{eq} in the sense of Definition \ref{apdef}. Then
$$
\theta \in L^{\infty}((0,\infty)\times(a,b)).
$$
\end{tw}
\begin{proof}
In the first step, we show that
$$
\theta\in{L^{\infty}(0,\infty;\lab 2)}.
$$
This step seems redundant since all the reasoning is contained in the iteration step, and iteration might be started at the level of the first norm of 
$\theta$ (estimated time-independently in Proposition \ref{energy_cons}). It is, however, illustrative, and the reader sees precisely what comes next. 
Let us multiply the lower equation in (\ref{eq}) by $\theta$ and integrate over $(a,b)$. Then
$$
\iab\theta_t\theta dx-\iab\theta_{xx}\theta dx=\mu\iab u_{tx}\theta^2dx.
$$
It leads to the identity
\begin{align}\label{thin2}
\frac12\ddt\iab\theta^2dx+\iab\theta^2_xdx=\mu\iab u_{tx}\theta^2dx.
\end{align}
We shall estimate the right-hand side of the above equality.
\begin{align*}%\label{thin1}
\mu\iab u_{tx}\theta^2dx=-2\mu\iab u_t\theta_x\theta dx\leq C\|u_t\|_{\lab2}\|\theta\|_{\lab\infty}\|\theta_x\|_{\lab2}.
\end{align*}
By \eqref{balance}, we know that $\|u_t\|_{\lab2}$ is bounded independently of time. Next, the norm $\|\theta\|_{\lab\infty}$ is bounded via the 
Gagliardo-Nirenberg inequality.
\begin{align*}
\|\theta\|_{\lab\infty}\leq C_1\|\theta_x\|_{\lab 2}^{\frac23}\|\theta\|_{\lab1}^{\frac13}+C_2\|\theta\|_{\lab1}
\end{align*}
Again, by \eqref{balance}, we know that $\|\theta\|_{\lab1}$ is bounded independently of time. Therefore, we have
\begin{align*}
\mu\iab u_{tx}\theta^2dx=-2\mu\iab u_t\theta_x\theta dx\leq C_1\|\theta_x\|^{\frac53}_{\lab2}+C_2\|\theta_x\|_{\lab2}.
\end{align*}
Now, we use the Young inequality to deduce
\begin{align*}
\mu\iab u_{tx}\theta^2dx\leq \frac12 \|\theta_x\|_{\lab2}^2+C.
\end{align*}
Plugging it into \eqref{thin2} yields
\begin{align}\label{thin3}
\ddt\|\theta\|_{\lab2}^2+\|\theta_x\|_{\lab2}^2\leq C.
\end{align}
Next use of the Gagliardo-Nirenberg inequality gives
$$
\|\theta\|_{\lab2}\leq C_1\|\theta_x\|_{\lab2}^{\frac13}\|\theta\|_{\lab1}^{\frac23}+C_2\|\theta\|_{\lab1}.
$$
Hence,
\begin{align*}
\|\theta\|_{\lab2}^2\leq C_1\|\theta_x\|_{\lab2}^2+C_2.
\end{align*}

Plugging the above inequality into \eqref{thin3}, we obtain
\begin{align*}
\ddt\|\theta\|_{\lab2}^2+C_1\|\theta\|_{\lab2}^2\leq C_2,
\end{align*}
where $C_1, C_2$ do not depend on time. Now, we immediately get
\begin{align*}
\|\theta\|_{\lab2}^2\leq \|\theta_0\|_{\lab2}^2+\frac{C_2}{C_1}.
\end{align*}
Hence,
\begin{align*}
\|\theta\|_{L^{\infty}(0,T;\lab2)}\leq C.
\end{align*}
Since $C$ is time-independent, we obtain
\begin{align*}
\|\theta\|_{L^{\infty}(0,\infty;\lab2)}\leq C.
\end{align*}

Now, we will estimate the norm $\|\theta\|_{L^{\infty}\left(0,\infty;\lab{2^{n+1}}\right)}$ for an arbitrary $n\in\n$.
Let us multiply the heat equation in \eqref{eq} by $\theta^{2^{n+1}-1}$ and integrate over $(a,b)$. After integration by parts, we have
\begin{align}\label{nin2}
\frac1{2^{n+1}}\ddt\iab\theta^{2^{n+1}}dx+\left(\frac1{2^{n-1}}-\frac1{4^n}\right)\iab\left[\left(\theta^{2^n}\right)_x\right]^2dx=\mu\iab 
u_{tx}\theta^{2^{n+1}}dx.
\end{align}
We estimate the right-hand side of the above equality
\begin{align}\label{nin1}
\mu\iab u_{tx}\theta^{2^{n+1}}dx&=-\mu 2^{n+1}\iab u_t\theta_x\theta^{2^{n+1}-1}dx=-2\mu\iab u_t\theta^{2^n}\left(\theta^{2^n}\right)_xdx\nonumber\\
&\leq C\|u_t\|_{\lab2}\|\left(\theta^{2^n}\right)_x\|_{\lab2}\|\theta^{2^n}\|_{\lab\infty}.
\end{align}
The Gagliardo-Nirenberg inequality yields
\begin{align*}
\|\theta^{2^n}\|_{\lab\infty}\leq C_1\|\left(\theta^{2^n}\right)_x\|_{\lab2}^{\frac23}\|\theta^{2^n}\|_{\lab1}^{\frac13}+C_2\|\theta^{2^n}\|_{\lab1}.
\end{align*}
Plugging in the above inequality into \eqref{nin1} we get
\begin{align*}
\mu\iab u_{tx}\theta^{2^{n+1}}dx&\leq 
C\left(\|\left(\theta^{2^n}\right)_x\|_{\lab2}^{\frac53}\|\theta^{2^n}\|_{\lab1}^{\frac13}+\|\left(\theta^{2^n}\right)_x\|_{\lab2}\|\theta^{2^n}\|_{\lab1}\right)\\
&\leq \epsilon \|\left(\theta^{2^n}\right)_x\|_{\lab2}^2+C(\epsilon^{-1}+\epsilon^{-5})\|\theta^{2^n}\|_{\lab1}^2,
\end{align*}
where $\epsilon$ is a positive number and will be defined soon.
Using \eqref{nin2}, we obtain
\begin{align*}
\frac1{2^{n+1}}\ddt\iab\theta^{2^{n+1}}dx+a_n\|\left(\theta^{2^n}\right)_x\|_{\lab2}^2\leq \epsilon 
\|\left(\theta^{2^n}\right)_x\|_{\lab2}^2+C(\epsilon^{-1}+\epsilon^{-5})\|\theta^{2^n}\|_{\lab1}^2,
\end{align*}
where we denoted $a_n:=\frac1{2^{n-1}}-\frac1{4^n}$. Let us take $\epsilon:=\frac{a_n}2$. Then, we have
\begin{align*}
\frac1{2^{n+1}}\ddt\iab\theta^{2^{n+1}}dx+\frac{a_n}2\|\left(\theta^{2^n}\right)_x\|_{\lab2}^2\leq C(a_n^{-1}+a_n^{-5})\|\theta^{2^n}\|_{\lab1}^2.
\end{align*}
We again use the Gagliardo-Nirenberg inequality and the Young inequality. We obtain
\begin{align*}
\frac1{2^{n+1}}\ddt\iab\theta^{2^{n+1}}dx+\frac{a_n}2\|\theta^{2^n}\|_{\lab2}^2\leq C(a_n^{-1}+a_n^{-5}+1)\|\theta^{2^n}\|_{\lab1}^2.
\end{align*}
We rewrite the above inequality as follows
\begin{align*}
\ddt\|\theta\|_{\lab{2^{n+1}}}^{2^{n+1}}+\|\theta\|_{\lab{2^{n+1}}}^{2^{n+1}}\leq C_1\|\theta\|_{\lab{2^n}}^{2^{n+1}},
\end{align*}
where $C(a_n^{-1}+a_n^{-5}+1)2^{n+1}\leq C 2^{6n}=C_1$. Consequently,
\begin{align*}
\|\theta\|_{\lab{2^{n+1}}}^{2^{n+1}}\leq (b-a)\|\theta_0\|_{\lab\infty}^{2^{n+1}}+C 
2^{6n}\|\theta\|_{L^{\infty}\left(0,\infty;\lab{2^n}\right)}^{2^{n+1}},
\end{align*}
which gives
\begin{align}\label{jeden}
\|\theta\|_{L^{\infty}\left(0,\infty;\lab{2^{n+1}}\right)}\leq\left(\tilde{c}\|\theta_0\|_{\lab\infty}^{2^{n+1}}+C 
2^{6n}\|\theta\|^{2^{n+1}}_{L^{\infty}\left(0,\infty;\lab{2^n}\right)}\right)^{\frac1{2^{n+1}}},
\end{align}
where $\tilde{c}:=\max\{1,b-a\}$.
Let us introduce a notation
\[
m_n:=\max\{\|\theta_0\|_{\lab{\infty}},\|\theta\|_{L^{\infty}\left(0,\infty;\lab{2^n}\right)}\}.
\]
Then, by \eqref{jeden}, we have
\begin{align*}
m_{n+1}\leq m_n(\tilde{c}+C2^{6n})^{\frac1{2^{n+1}}}\leq m_n C^{\frac1{2^{n+1}}}2^{\frac{6n}{2^{n+1}}}.
\end{align*}
Thus,
\begin{align*}
m_n\leq m_1 \prod_{k=2}^{n}C^{\frac1{2^k}}2^{\frac{6(k-1)}{2^k}}=m_1C^{\sum_{k=2}^n\frac1{2^k}}2^{\sum_{k=1}^n\frac{6(k-1)}{2^k}}\leq m_1 
C^{\sum_{k=2}^{\infty}\frac1{2^k}}2^{\sum_{k=1}^{\infty}\frac{6(k-1)}{2^k}}.
\end{align*}
Both series in exponents in the above inequality are convergent, so the sequence $m_n$ is bounded. Hence, we have
\begin{align*}
\|\theta\|_{L^{\infty}\left(0,\infty,\lab{2^n}\right)}\leq C
\end{align*}
for all $n\in\n$, and the claim is shown.
\end{proof}

The above theorem and Proposition \ref{ogrzpop} give us immediately the following result, crucial in the compactness estimates.
\begin{cor}\label{oszcor}
Let $\theta$ be a positive weak solution of \eqref{eq} in the sense of Definition \ref{apdef}. Then,
\begin{align*}
\int_0^{\infty}\iab\frac{\theta^2_x}{\theta}dxdt<\infty.
\end{align*}
\end{cor}
\begin{proof}
We have
\begin{align*}
\int_0^{\infty}\iab\frac{\theta_x^2}{\theta}dxdt=\int_0^{\infty}\iab\theta\frac{\theta_x^2}{\theta^2}dxdt\leq\|\theta\|_{L^{\infty}((0,\infty)\times(a,b))}\int_0^{\infty}\iab\tau_x^2dxdt.
\end{align*}
We conclude using Theorem \ref{oszth} and \eqref{taul2}.
\end{proof}

By Theorem \ref{main_theorem}, we know that $\theta$ is positive. To be more precise, in \cite{BC}, we have shown that for each $T>0$, there exists 
$m_T>0$ such that $\theta \geq m_T$ on $(0,T)$. The following theorem shows that actually, the stronger claim holds, $m$ can be chosen independent of 
time, i.e., the temperature is bounded away from $0$.

\begin{tw}\label{Twoddolu}
Let $\theta$ be a positive temperature solving \eqref{eq} in the sense of Definition \ref{apdef}. Then,
\begin{align*}
\tau:=\log\theta\in L^{\infty}((0,\infty)\times(a,b)).
\end{align*}
\end{tw}
\begin{proof}
Theorem \ref{oszth} gives
\begin{align*}
\tau\leq \log\|\theta\|_{L^{\infty}((0,\infty)\times(a,b))}
\end{align*}
for almost all $(t,x)\in(0,\infty)\times(a,b)$. Thus, it is sufficient to show that
\begin{align*}
\tau^-\in L^{\infty}((0,\infty)\times(a,b)),
\end{align*}
where $\tau^-:=\max\{0,-\tau\}$.

We take the first step and estimate the $L^2$ norm of $\tau^-$. We need to do it since we shall start the iteration from the $L^2$ norm.

Multiplying \eqref{taurownanie} by $-\tau^-$ and integrating over $[a,b]$, we arrive at
\begin{align*}
\frac12\ddt\iab(\tau^-)^2dx+\iab[(\tau^-)_x]^2dx+\iab\tau_x^2\tau^-dx=-\mu\iab u_{tx}\tau^-dx.
\end{align*}
This gives us immediately that
\begin{align}\label{login1}
\frac12\ddt\|\tau^-\|_{\lab{2}}^2+\|(\tau^-)_x\|_{\lab2}^2\leq-\mu\iab u_{tx}\tau^-dx.
\end{align}
We shall bound the right-hand side of the above inequality.
\begin{align*}
-\mu\iab u_{tx}\tau^-dx=\mu\iab u_t(\tau^-)_xdx\leq|\mu|\|u_t\|_{\lab2}\|(\tau^-)_x\|_{\lab2}
\end{align*}
We put this result into \eqref{login1} and use the Young inequality. This leads to
\begin{align}\label{login2}
\ddt\|\tau^-\|_{\lab{2}}^2+\|(\tau^-)_x\|_{\lab2}^2\leq C.
\end{align}

From \eqref{tauosz} we know that $\|\tau^-\|_{\lab1}$ is bounded for almost all $t\in [0,\infty)$. By the Gagliardo-Nirenberg inequality, we have
\begin{align*}
\|\tau^-\|_{\lab2}\leq  &C_1\|(\tau^-)_x\|_{\lab2}^{\frac13}\|\tau^-\|_{\lab1}^{\frac23}+C_2\|\tau^-\|_{\lab1}\\
\leq & C_1\|(\tau^-)_x\|_{\lab2}+C_2\|\tau^-\|_{\lab1},
\end{align*}
where in the last step, we used the Young inequality. Now, plugging the last inequality into \eqref{login2}, we obtain
\begin{align*}
\ddt\|\tau^-\|_{\lab{2}}^2+C_1\|\tau^-\|_{\lab2}^2\leq C_2.
\end{align*}
Thus, we get that
\begin{align*}
\tau^-\in L^{\infty}(0,\infty; \lab2)
\end{align*}

Next, Let us take $n\in\n$. We multiply equation \eqref{taurownanie} by $-(\tau^-)^{2^{n+1}-1}$ and then integrate over $[a,b]$. It yields
\begin{align}\label{login4}
\frac{1}{2^{n+1}}\ddt\iab (\tau^-)^{2^{n+1}}dx+\left(\frac1{2^{n-1}}-\frac1{4^{n}}\right)\iab\left[\left((\tau^-)^{2^n}\right)_x\right]^2\leq 
-\mu\iab(\tau^-)^{2^{n+1}-1}u_{tx}dx.
\end{align}
We calculate the right-hand side of the above inequality.
\begin{align}
-\mu\iab(\tau^-)^{2^{n+1}-1}u_{tx}dx&=\mu(2^{n+1}-1)\iab(\tau^-)_x(\tau^-)^{2^{n+1}-2}u_tdx\nonumber\\
&=\mu\left(2-\frac1{2^n}\right)\iab[(\tau^-)^{2^n}]_x(\tau^-)^{2^{n}-1}u_tdx\nonumber\\
&\leq |\mu|\left(2-\frac1{2^n}\right)\|u_t\|_{\lab2}\|[(\tau^-)^{2^n}]_x\|_{\lab2}\|(\tau^-)^{2^n}\|_{\lab{\infty}}^{1-\frac1{2^n}}\nonumber\\
&\leq|\mu|2\|u_t\|_{\lab2}\|[(\tau^-)^{2^n}]_x\|_{\lab2}(1+\|(\tau^-)^{2^n}\|_{\lab{\infty}})\label{login3}
\end{align}
Now, we use the Gagliardo-Nirenberg inequality
\begin{align*}
\|(\tau^-)^{2^n}\|_{\lab{\infty}}\leq C_1\|[(\tau^-)^{2^n}]_x\|_{\lab2}^{\frac23}\|(\tau^-)^{2^n}\|_{\lab1}^{\frac13}+C_2\|(\tau^-)^{2^n}\|_{\lab1}.
\end{align*}
Plugging the obtained inequality into \eqref{login3}, we arrive at
\begin{align*}
-\mu\iab&(\tau^-)^{2^{n+1}-1}u_{tx}dx\\
&\leq C(\|[(\tau^-)^{2^n}]_x\|_{\lab2}+\|[(\tau^-)^{2^n}]_x\|_{\lab2}^{\frac53}\|(\tau^-)^{2^n}\|_{\lab1}^{\frac13}\\
&+\|[(\tau^-)^{2^n}]_x\|_{\lab2}\|(\tau^-)^{2^n}\|_{\lab1})\\
&\leq \epsilon \|[(\tau^-)^{2^n}]_x\|_{\lab2}^2+C\epsilon^{-1}+C(\epsilon^{-1}+\epsilon^{-5})\|(\tau^-)^{2^n}\|_{\lab1}^2.
\end{align*}

We plug this into \eqref{login4} and obtain
\begin{align*}
\frac{1}{2^{n+1}}\ddt&\iab (\tau^-)^{2^{n+1}}dx+\left(\frac1{2^{n-1}}-\frac1{4^{n}}\right)\iab\left[\left((\tau^-)^{2^n}\right)_x\right]^2\\
&\leq \epsilon \|[(\tau^-)^{2^n}]_x\|_{\lab2}^2+C\epsilon^{-1}+C(\epsilon^{-1}+\epsilon^{-5})\|(\tau^-)^{2^n}\|_{\lab1}^2.
\end{align*}
We denote $a_n:=\frac1{2^{n-1}}-\frac1{4^{n}}$ and we take $\epsilon:=\frac{a_n}2$. Thus, we get
\begin{align*}
\frac{1}{2^{n+1}}\ddt&\iab (\tau^-)^{2^{n+1}}dx+\frac{a_n}2\iab\left[\left((\tau^-)^{2^n}\right)_x\right]^2\\
&\leq C a_n^{-1}+C(a_n^{-1}+a_n^{-5})\|(\tau^-)^{2^n}\|_{\lab1}^2.
\end{align*}

Now, we again use the Gagliardo-Nirenberg inequality and the Young inequality. It yields
\begin{align*}
\frac{1}{2^{n+1}}\ddt&\iab (\tau^-)^{2^{n+1}}dx+\frac{a_n}2\|(\tau^-)^{2n}\|_{\lab2}^2\\
&\leq C a_n^{-1}+C(a_n^{-1}+a_n^{-5}+1)\|(\tau^-)^{2^n}\|_{\lab1}^2.
\end{align*}
We rewrite the above inequality in the following way
\begin{align*}
\ddt&\|\tau^-\|_{\lab{2^{n+1}}}^{2^{n+1}}+\|\tau^-\|_{\lab{2^{n+1}}}^{2^{n+1}}\leq C2^{6n}(\|\tau^-\|_{\lab{2^n}}^{2^{n+1}}+1).
\end{align*}
It leads us to
\begin{align*}
\|\tau^-\|_{L^{\infty}\left(0,\infty;\lab{2^{n+1}}\right)}^{2^{n+1}}\leq 
C2^{6n}(\|\tau^-\|_{L^{\infty}\left(0,\infty;\lab{2^n}\right)}^{2^{n+1}}+1)+(b-a)\|\tau_0^-\|_{\lab{\infty}}^{2^{n+1}}
\end{align*}
for all $n\in\n$ and for almost all $t\in [0,\infty)$, where $\tau_0:=\log\theta_0$. Now, we immediately obtain
\begin{align}\label{login5}
\|\tau^-\|_{L^{\infty}\left(0,\infty;\lab{2^{n+1}}\right)}\leq ( 
C2^{6n}(\|\tau^-\|_{L^{\infty}\left(0,\infty;\lab{2^n}\right)}^{2^{n+1}}+1)+\tilde{c}\|\tau_0^-\|_{\lab{\infty}}^{2^{n+1}})^{\frac1{2^{n+1}}},
\end{align}
where $\tilde{c}=\max\{1,b-a\}$. We can also assume that $C\geq 1$. Let us introduce a sequence 
$$m_n:=\max\{\|\tau^-\|_{L^{\infty}\left(0,\infty;\lab{2^n}\right)},\|\tau_0^-\|_{\lab{\infty}},1\}.$$ Inequality \eqref{login5} gives us
\begin{align*}
m_{n+1}\leq ( C2^{6n}+\tilde c)^{\frac{1}{2^{n+1}}} m_n\leq C^{\frac1{2^{n+1}}}2^{\frac{6n}{2^{n+1}}}m_n.
\end{align*}
It is an analogous inequality to the one appearing in the proof of Theorem \ref{oszth}. Thus, similarly, we obtain
\begin{align*}
\tau^-\in L^{\infty}((0,\infty)\times (a,b)).
\end{align*}
\end{proof}

The above theorem gives that for almost all $(t,x)\in (0,\infty)\times (a,b)$  we have
\begin{align*}
-\|\tau\|_{L^{\infty}((0,\infty)\times (a,b))}\leq\log\theta\leq\|\tau\|_{L^{\infty}((0,\infty)\times (a,b))}.
\end{align*}
It yields that for almost all $(t,x)\in (0,\infty)\times (a,b)$  one has
\begin{align*}
0<\exp\left(-\|\tau\|_{L^{\infty}((0,\infty)\times (a,b))}\right)\leq \theta\leq\|\theta\|_{L^{\infty}((0,\infty)\times (a,b))}.
\end{align*}

\section{Time-independent estimates for higher order derivatives of the displacement and the temperature}\label{compactness}

In order to pass to the limit over the time sequences, we need compactness estimates. In what follows, we present the time-independent estimates of 
higher-order derivatives of displacement and temperature, guaranteeing required compactness.

We derive the estimates on the approximation sequence from Proposition \ref{aproxex}. Next, the passage to the limit shall yield the exactly same 
estimates of the solutions to \eqref{eq}.
\begin{tw}\label{zwartosc}
Let $(u_n,\theta_n)$ be the solution of the approximate problem in the sense of Definition \ref{apdef} for any $n\in\n$. Then, for any $t>0$ and all 
$n\in\n$ we have
\begin{align}\label{appep}
\|\theta_{n,x}\|_{\lab2}^2+\|u_{n,tx}\|_{\lab2}^2+\|u_{n,xx}\|^2_{\lab2}\leq C.
\end{align}
The constant $C$ depends on $b-a,\mu,\|\theta_0\|_{\hab},\tilde\theta,\|u_0\|_{H^2(a,b)}$ and  $\|v_0\|_{\habo}$.
\end{tw}
\begin{proof}
For simplicity, we skip the subscripts in $u_n$ and $\theta_n$ in the proof. By Proposition \ref{aproxex}, we know that $u$ and $\theta$ are regular 
enough to apply Lemma \ref{estlem}. Hence, we obtain
\begin{align}\label{ballep}
\frac12\ddt\left(\int_a^b\frac{\theta_x^2}{\theta}dx+\int_a^bu_{tx}^2dx+\int_a^bu_{xx}^2dx\right)=-\int_a^b\theta\left[(\log\theta)_{xx}\right]^2dx+\frac{\mu}2\int_a^b\frac{\theta^2_x}{\theta}u_{tx}dx.
\end{align}
We estimate the second term on the right-hand side of \eqref{ballep}.
\begin{align}\label{nier1}
\nonumber\iab\frac{\theta^2_x}{\theta}u_{tx}dx&=4\iab\left(\theta^{\ul}\right)^2_xu_{tx}dx=-8\iab\left(\theta^{\ul}\right)_x\left(\theta^{\ul}\right)_{xx}u_tdx\\
\nonumber&\leq 8\|u_t\|_{\lab{\infty}}\left\|\left(\theta^{\ul}\right)_{xx}\right\|_{\lab2}\left\|\left(\theta^{\ul}\right)_{x}\right\|_{\lab2}\\
&\leq 
\epsilon\left\|\left(\theta^{\ul}\right)_{xx}\right\|^2_{\lab2}+\frac{1}{16\epsilon}\left\|\left(\theta^{\ul}\right)_{x}\right\|^2_{\lab2}\|u_t\|_{\lab{\infty}}^2.
\end{align}
We apply the Gagliardo-Nirenberg inequality to the norm $\|u_t\|_{\lab{\infty}}$ and obtain
\begin{align*}
\|u_t\|_{\lab{\infty}}\leq C\|u_{tx}\|_{\lab2}^{\ul}\|u_t\|_{\lab2}^{\ul}.
\end{align*}

Taking this into account in \eqref{nier1}, together with Lemma \ref{cfhslem}, we obtain
\begin{align}\label{dwa}
\iab\frac{\theta^2_x}{\theta}u_{tx}dx&\leq 
\epsilon\frac{13}{8}\iab\theta\left[(\log\theta)_{xx}\right]^2dx+C\left\|\left(\theta^{\ul}\right)_{x}\right\|^2_{\lab2}\|u_{tx}\|_{\lab2}\|u_t\|_{\lab2}\nonumber\\
&\leq 
\epsilon\frac{13}{8}\iab\theta\left[(\log\theta)_{xx}\right]^2dx+C\left\|\left(\theta^{\ul}\right)_{x}\right\|^2_{\lab2}\|u_{tx}\|_{\lab2}\nonumber\\
&\leq 
\epsilon\frac{13}{8}\iab\theta\left[(\log\theta)_{xx}\right]^2dx+C_1\left\|\left(\theta^{\ul}\right)_{x}\right\|^2_{\lab2}+C_2\left\|\left(\theta^{\ul}\right)_{x}\right\|^2_{\lab2}\|u_{tx}\|_{\lab2}^2,
\end{align}
where constant $C$ contains $\|u_t\|_{L^2(a,b)}$. The last inequality might seem strange at the first sight. It is, however, crucial in order to derive 
the proper tricky Gr\"{o}nwall-type lemma yielding a time-independent estimate.

We plug in \eqref{dwa} in \eqref{ballep} and choose $\epsilon$ such that $\epsilon\frac{13}{8}-1\leq 0$, this gives
\begin{align}\label{trzy}
\frac12\ddt&\left(\int_a^b\frac{\theta_x^2}{\theta}dx+\int_a^bu_{tx}^2dx+\int_a^bu_{xx}^2dx\right)\nonumber\\ &\leq 
C_1\left\|\left(\theta^{\ul}\right)_{x}\right\|^2_{\lab2}+C_2\left\|\left(\theta^{\ul}\right)_{x}\right\|^2_{\lab2}\|u_{tx}\|_{\lab2}^2.
\end{align}

Introducing the notation $x(t):=\int_a^b\frac{\theta_x^2}{\theta}dx$  and $y(t):=\int_a^bu_{tx}^2dx+\int_a^bu_{xx}^2dx$, we rewrite \eqref{trzy} in the 
following form
\begin{align*}
x'(t)+y'(t)\leq C_1x(t)+C_2x(t)y(t)\leq C_1x(t)+C_2\left(x(t)y(t)+x(t)^2\right).
\end{align*}
Multiplying both sides by $\exp\left(-C_2\int_0^tx(s)ds\right)$ yields
\begin{align*}
\ddt\left(\exp\left(-C_2\int_0^tx(s)ds\right)\left(x(t)+y(t)\right)\right)\leq C_1\exp\left(-C_2\int_0^tx(s)ds\right) x(t).
\end{align*}
We integrate over $[0,t]$ for arbitrary $t$ and get
\begin{align*}
\exp\left(-C_2\int_0^tx(s)ds\right)&\left(x(t)+y(t)\right) \\
&\leq C_1\int_0^t\exp\left(-C_2\int_0^sx(u)du\right) x(s)ds+x(0)+y(0),
\end{align*}
which immediately gives
\begin{align*}
x(t)+y(t)\leq \exp\left(C_2\int_0^tx(s)ds\right)\left(C_1\int_0^tx(s)ds+x(0)+y(0)\right)
\end{align*}
for all $t\in[0,\infty)$.

Notice that the estimate does not depend on time in light of Corollary \ref{oszcor}. Hence, we have
\begin{align*}
x(t)+y(t)\leq \exp\left(C_2\int_0^{\infty}x(s)ds\right)\left(C_1\int_0^{\infty}x(s)ds+x(0)+y(0)\right)
\end{align*}
for all $t\in[0,\infty)$.
Next, we notice that in light of Theorem \ref{oszth}, we additionally have
\begin{align*}
\|\theta_x\|_{\lab2}^2\leq\|\theta\|_{L^{\infty}((0,\infty)\times(a,b))}\iab\frac{\theta_x^2}{\theta}dx=\|\theta\|_{L^{\infty}((0,\infty)\times(a,b))}x(t)
\end{align*}
for all $t\in [0,\infty)$.
The proof is finished.
\end{proof}

In light of Theorem \ref{zwartosc}, in particular since $C$ in \eqref{appep} does not depend on $n$, we know that $(u_n,\theta_n)$ is weakly convergent 
(possibly on a subsequence) to $(u,\theta)$, which is a unique weak solution of \eqref{eq}. Hence, $(u,\theta)$ solving \eqref{eq}, satisfy
\begin{align}\label{beest}
\begin{split}
u&\in L^{\infty}(0,\infty;H^2(a,b))\cap W^{1,\infty}(0,\infty;\habo),\\
\theta&\in L^{\infty}(0,\infty;\hab).
\end{split}
\end{align}

\section{Displacement and temperature at infinity}%\label{last_section}

The present section is devoted to the proof of the main Theorem \ref{main_tw_gr}. We classify the asymptotic behavior of a heated string. Although the 
strong damping mechanisms like heat radiation and material damage due to heat or resistance of the air are not taken into account in our model, it still 
turns out that the heat propagation itself is a strong stabilizing effect, ensuring a string decay. It is a bit surprising since the natural behavior 
of a string is the oscillatory one. In our argument, we see that the second principle of thermodynamics is responsible for the vibration suppression.

Technically, our proof is based on the estimates in Sections \ref{iteration} and \ref{compactness}. However, we still need a stability of the solution 
map with respect to initial conditions. This is the content of Proposition \ref{ciagzallip}. Next, we define the dynamical system and its phase space 
corresponding to the evolution of solutions of \eqref{eq}. We are then in a position to identify the $\omega$-limit sets. One of the main tools in our 
argument is a quantitative version of the second principle of thermodynamics stated in Proposition \ref{ogrzpop}. A reader following closely our proof 
will notice the role of the second principle of thermodynamics in the identification of the asymptotic states of a heated string.

We start by stating the stability with respect to initial data in proper functional spaces.
\begin{prop}\label{ciagzallip}
Let us take two triples of the initial values $u_{0,1},u_{0,2}\in H^2(a,b)\cap\habo, v_{0,1},v_{0,2}\in\habo$ and $\theta_{0,1},\theta_{0,2}\in\hab$. 
Let us also assume that $(u_1,\theta_1)$ is a weak solution of \eqref{eq} starting from $u_{0,1},v_{0,1},\theta_{0,1}$ and $(u_2,\theta_2)$ is a weak 
solution starting from $u_{0,2},v_{0,2},\theta_{0,2}$. Moreover, we assume that there exists $M>0$ such that
 $$\|u_{1,tx}\|_{L^{\infty}\left(0,\infty;\lab2\right)}+\|\theta_2\|_{L^{\infty}\left(0,\infty;\hab\right)}\leq M\footnote{The inequality is satisfied 
 by solutions of \eqref{eq} due to \eqref{beest}.}.$$

Then, for all $t\in(0,\infty)$ there exists the constant $C>0$  such that
\begin{align*}
&\|u_1(t,\cdot)-u_2(t,\cdot)\|_{\habo}+\|u_{1,t}(t,\cdot)-u_{2,t}(t,\cdot)\|_{\lab2}+\|\theta_1(t,\cdot)-\theta_2(t,\cdot)\|_{\lab2}\\
&\quad\leq C\left(\|u_{0,1}-u_{0,2}\|_{\habo}+\|v_{0,1}-v_{0,2}\|_{\lab2}+\|\theta_{0,1}-\theta_{0,2}\|_{\lab2}\right)
\end{align*}
and $C=C(t,M,|\mu|)$.
\end{prop}
\begin{proof}
Let us denote $u:=u_1-u_2$ and $\theta:=\theta_1-\theta_2$. We subtract the equations for $(u_2,\theta_2)$ from the equations for $(u_1,\theta_1)$ and 
obtain
\begin{align}\label{uklpom}
\begin{cases}
u_{tt}-u_{xx}=\mu\theta_x,\\
\theta_t-\theta_{xx}=\mu\theta u_{1,tx}+\mu\theta_2u_{tx}.
\end{cases}
\end{align}
We multiply by $u_t$ the first of the above equations and arrive at
\begin{align}\label{glrow1}
\frac12\ddt\left(\iab u_t^2dx+\iab u_x^2dx \right)=\mu\iab u_t\theta_x\leq C\|u_t\|_{\lab2}^2+\frac14\|\theta_x\|^2_{\lab2}.
\end{align}

Whereas, multiplying the second equation in \eqref{uklpom} by $\theta$, we obtain
\begin{align}\label{glrow2}
\nonumber\frac12\ddt&\iab\theta^2 dx+\iab\theta^2_x dx=\mu\iab\theta^2u_{1,tx}dx+\mu\iab\theta\theta_2u_{tx}dx\\
&=\mu\iab\theta^2u_{1,tx}dx-\mu\iab\theta_x\theta_2u_{t}dx-\mu\iab\theta\theta_{2,x}u_{t}dx=I_1+I_2+I_3.
\end{align}
We estimate integrals $I_1,I_2$ and $I_3$.
\begin{align}\label{row21}
I_1&\leq|\mu|\|\theta\|_{\lab2}\|\theta\|_{\lab{\infty}}\|u_{1,tx}\|_{\lab2}\leq |\mu|M 
C\|\theta\|_{\lab2}\left(\|\theta\|_{\lab2}+\|\theta_x\|_{\lab2}\right)\nonumber\\
&\leq C\|\theta\|_{\lab2}^2+\frac1{12}\|\theta_x\|_{\lab2}^2,
\end{align}
\begin{align}\label{row22}
\nonumber I_2&\leq |\mu|\|\theta_x\|_{\lab2}\|\theta_2\|_{\lab{\infty}}\|u_t\|_{\lab2}\leq 
|\mu|C\|\theta_x\|_{\lab2}\|\theta_2\|_{\hab}\|u_t\|_{\lab2}\\
&\leq\frac1{12}\|\theta_x\|_{\lab2}^2+C\|u_t\|_{\lab2}^2,
\end{align}
and
\begin{align}\label{row23}
\nonumber I_3&\leq |\mu|\|\theta\|_{\lab{\infty}}\|\theta_{2,x}\|_{\lab2}\|u_t\|_{\lab2}\leq |\mu|M 
C\left(\|\theta\|_{\lab2}+\|\theta_x\|_{\lab2}\right)\|u_t\|_{\lab2}\\
&\leq\frac1{12}\|\theta_x\|^2_{\lab2}+C\left(\|\theta\|^2_{\lab2}+\|u_t\|^2_{\lab2}\right).
\end{align}
Next, inequalities \eqref{row21}, \eqref{row22} and \eqref{row23} are applied in \eqref{glrow2}. We have
\begin{align*}
\frac12\ddt&\iab\theta^2 dx+\iab\theta^2_x dx\leq\frac14\|\theta_x\|^2_{\lab2}+C\left(\|\theta\|^2_{\lab2}+\|u_t\|^2_{\lab2}\right).
\end{align*}

We add the above inequality and \eqref{glrow1}. It yields
\begin{align*}
\frac12\ddt&\left(\iab u_t^2dx+\iab u_x^2dx+\iab\theta^2 dx\right)+\iab\theta^2_x 
dx\leq\frac12\|\theta_x\|^2_{\lab2}+C\left(\|\theta\|^2_{\lab2}+\|u_t\|^2_{\lab2}\right).
\end{align*}
It immediately implies
\begin{align*}
\ddt&\left(\iab u_t^2dx+\iab u_x^2dx+\iab\theta^2 dx\right)\leq C\left(\|\theta\|^2_{\lab2}+\|u_t\|^2_{\lab2}+\|u_x\|^2_{\lab2}\right).
\end{align*}
The Gr\"{o}nwall inequality finishes the proof.
\end{proof}
In the next step of the proof of Theorem \ref{main_tw_gr}, we define the corresponding phase space and the dynamical system.
We denote
\begin{align*}
\M:=\left\{\eta\in \hab\colon \operatorname*{ess\, inf}_{x\in(a,b)}\eta(x)>0\right\}.
\end{align*}
Let us take $t\geq 0$. We define an operator
\begin{align*}
S(t):\left(\habo\cap\habl\right)\times\habo\times\M\to \left(\habo\cap\habl\right)\times\habo\times\M
\end{align*}
by formula
\begin{align*}
S(t)(u_0,v_0,\theta_0)=(u(t),u_t(t),\theta(t)),
\end{align*}
where $u$ and $\theta$ are weak solutions of \eqref{eq} with initial values $u_0, v_0$ and $\theta_0$. Thanks to the uniqueness of weak solutions of 
\eqref{eq}, we obtain
\begin{align*}
S(t+h)=S(h)\circ S(t)
\end{align*}
for all $t,h\in [0,\infty)$.

Proposition \ref{ciagzallip} implies the following important corollary.
\begin{cor}\label{ciagzalcor}
Let  $(u_0,v_0,\theta_0),(u_n,v_n,\theta_n)\in \left(\habo\cap\habl\right)\times\habo\times\M$ be such that \begin{align*}
(u_n,v_n,\theta_n)\to(u_0,v_0,\theta_0)\textrm{ in }\habo\times\lab2\times\lab2.
\end{align*}
Then also
\begin{align*}
S(t)(u_n,v_n,\theta_n)\to S(t)(u_0,v_0,\theta_0)\textrm{ in }\habo\times\lab2\times\lab2
\end{align*}
for each $t\geq 0$.
\end{cor}

Now, we are ready to prove the paper's main theorem.
\begin{proof}[The proof of Theorem \ref{main_tw_gr}]
Let us take a sequence $t_n\in [0,\infty)$ such that $t_n\to\infty$. We consider the sequences $u(t_n,\cdot)$, $u_t(t_n,\cdot)$ and 
$\theta(t_n,\cdot)$. By Definition \ref{weakdef} and Remark \ref{rem1}, we know that
\begin{align}\label{cont}
\begin{split}
u\in C([0,\infty);&\habo),\ u_t \in C([0,\infty);L^2(a,b)),\\
\theta& \in C([0,\infty);\hab).
\end{split}
\end{align}
Due to \eqref{beest}, we see that an even stronger claim is satisfied by the sequence $\theta(t_n,\cdot)$, it is bounded in $\hab$. However, 
$u(t_n,\cdot)$ and $u_t(t_n,\cdot)$ are only essentially bounded in $\habl$ and $\habo$ respectively. It makes the compactness argument presented below 
a bit more delicate. Indeed, we only know that
\begin{align*}
u&\in L^{\infty}(0,\infty;H^2(a,b)),\quad u_t\in L^{\infty}(0,\infty,\habo).
\end{align*}
Thus, there exists a measurable set $\T\subset [0,\infty)$ such that $\left|[0,\infty)\setminus\T\right|=0$ and $u|_{\T}$, $u_t|_{\T}$ are bounded 
functions with values in $\habl$ and $\habo$ respectively.
However, if we show
\begin{align*}
\lim_{\substack{t\to\infty\\t\in\T}}u(t,\cdot)=0\textrm{ in }\habo
\end{align*}
and
\begin{align*}
\lim_{\substack{t\to\infty\\t\in\T}}u_t(t,\cdot)=0\textrm{ in }\lab2,
\end{align*}
then \eqref{cont} will imply that
\begin{align*}
\lim_{t\to\infty}u(t,\cdot)=0\textrm{ in }\habo,
\end{align*}
and
\begin{align*}
\lim_{t\to\infty}u_t(t,\cdot)=0\textrm{ in }\lab2.
\end{align*}
The details are left to the reader.

Therefore, we assume that $u(t_n,\cdot)$ and $u_t(t_n,\cdot)$ are bounded in $\habl$ and $\habo$ respectively. We also know that $\theta(t_n,\cdot)$ is 
bounded in $\hab$. Thus, we have a subsequence of $t_n$ (still denoted as $t_n$) such that
\begin{align}\label{cztery}
u(t_n,\cdot)&\rightharpoonup \tilde{u}\textrm{ in }\habl,\nonumber\\
u_t(t_n,\cdot)&\rightharpoonup\tilde{v}\textrm{ in }\habo,\\
\theta(t_n,\cdot)&\rightharpoonup\tilde{\theta}\textrm{ in }\hab,\nonumber
\end{align}
where $\tilde{u}\in\habo\cap\habl$, $\tilde v\in\habo$ and $\tilde\theta\in\hab$ are certain functions. By compact embeddings, we know that
\begin{align*}%\label{conv}
\begin{split}
u(t_n,\cdot)&\to \tilde{u}\textrm{ in }\habo,\\
u_t(t_n,\cdot)&\to\tilde{v}\textrm{ in }\lab2,\\
\theta(t_n,\cdot)&\to\tilde{\theta}\textrm{ in }\lab2.
\end{split}
\end{align*}
Next, we define solutions starting from functions, which are obtained as limits of $u(t_n,\cdot),\theta(t_n,\cdot)$. We denote by $\bar u$ and 
$\bar\theta$ the weak solution of \eqref{eq} with the initial values $\tilde{u}$, $\tilde v$ and $\tilde\theta$ (they can be defined due to the fact 
that the limiting objects have enough regularity to define weak solutions, see \eqref{cztery}).

Next, we take $h>0$ and consider the sequences $u(t_n+h,\cdot)$, $u_t(t_n+h,\cdot)$ and $\theta(t_n+h,\cdot)$. Similarly, as above, there exists 
$\hat{u}\in\habo\cap\habl$, $\hat v\in\habo$ and $\hat\theta\in\hab$ such that
\begin{align*}
u(t_n+h,\cdot)&\to \hat{u}\textrm{ in }\habo,\\
u_t(t_n+h,\cdot)&\to\hat{v}\textrm{ in }\lab2,\\
\theta(t_n+h,\cdot)&\to\hat{\theta}\textrm{ in }\lab2.
\end{align*}
If necessary, we take the subsequence of $t_n$ in the above convergences. By Corollary \ref{ciagzalcor}, we obtain
\begin{align*}
(\hat u,\hat 
v,\hat\theta)&=\lim_{n\to\infty}(u(t_n+h,\cdot),u_t(t_n+h,\cdot),\theta(t_n+h,\cdot))=\lim_{n\to\infty}S(h)(u(t_n,\cdot),u_t(t_n,\cdot),\theta(t_n,\cdot))\\
&=S(h)(\tilde u,\tilde v,\tilde\theta)=(\bar u(h,\cdot),\bar u_t(h,\cdot),\bar\theta(h,\cdot)).
\end{align*}

On the other hand, by Proposition \ref{ogrzpop}, we know that a function
\begin{align*}
t\mapsto\iab\log\theta(t,x)dx
\end{align*}
is non-decreasing. Thus, sequences $\iab\log\theta(t_n,x)dx$ and $\iab\log\theta(t_n+h,x)dx$ must be convergent to the same number. Moreover,
\begin{align}\label{piec}
\iab\log\tilde\theta dx=\lim_{n\to\infty}\iab\log\theta(t_n,x)dx=\lim_{n\to\infty}\iab\log\theta(t_n+h,x)dx=\iab\log\bar\theta(h,x)dx.
\end{align}
Indeed, $\theta(t_n,x)$ is convergent in $L^2$, so also a.e. (up to a subsequence). Next, we use Theorem \ref{Twoddolu} to infer that both 
$\bar{\theta}(h,\cdot)$ and $\tilde{\theta}$ are bounded away from $0$. The latter allows us to conclude \eqref{piec} via the Lebesgue theorem.

Next, we remark that \eqref{piec} states that
\begin{align*}
t\mapsto\iab\log\bar\theta(t,x)dx
\end{align*}
is constant. The second principle of thermodynamics has just told us that the integral of entropy is constant at the $\omega$-limit set. Further, we 
make use of the knowledge of the dissipation rate of the entropy in Proposition \ref{ogrzpop} to obtain
\begin{align*}
0=\ddt\iab\log\bar\theta dx=\iab\frac{\bar\theta_x^2}{\bar\theta^2}dx.
\end{align*}
Hence,
\[
\left(\log{\bar{\theta}}\right)_x=0\;\;\mbox{in}\;\;(a,b)
\]
for any $t>0$. So, for each fixed $t>0$, $\bar{\theta}$ is constant in space.
But then, applying Proposition \ref{ogrzpop} (II second principle of thermodynamics) again, we obtain
\begin{align*}
0=\ddt\iab\log\bar\theta dx=\ddt\log\bar\theta (b-a)=\frac{\bar\theta_t}{\bar\theta}(b-a).
\end{align*}
It means that $\bar\theta$ is also constant in time. So, $\bar{\theta}(t,x)$ is constant in space and time.

So far, we identified all the trajectories of \eqref{eq} starting from functions $\tilde{\theta}$. The functions $\tilde{\theta}$ are all the possible limits with time of a 
temperature satisfying \eqref{eq}. We observe now that, actually, the $\omega$-limit set of a solution to \eqref{eq} contains only one potential limit 
for temperature $\tilde{\theta}$. To this end, we utilize the second principle of thermodynamics once more. Assume there are two different functions 
$\tilde{\theta}_1$ and $\tilde{\theta}_2$, which are limits of $\theta(t_{n_k})$ and $\theta(t_{n_l})$ respectively, for different time sequences 
$t_{n_k}, t_{n_l}\rightarrow \infty$. But Proposition \ref{ogrzpop} says that
\[
\iab \log \left(\tilde{\theta}_1\right)dx =\iab \log \left(\tilde{\theta}_2\right)dx.
\]
Moreover, we already know that solutions starting from $\tilde{\theta}_1$ and $\tilde{\theta}_2$ are constant in space and time. Hence
\[
\log (\tilde{\theta}_1)=\log (\tilde{\theta}_2)
\]
and so $\tilde{\theta}_1=\tilde{\theta}_2$. Thus, we have just shown that the $\omega$-limit set of a solution to \eqref{eq} contains a single function 
$\tilde{\theta}$, which is constant. Next, we identify all potential limits of displacements.

First, we name $\tilde{u}$ the limit of $u(t_n,x)$ for a sequence $t_n\rightarrow \infty$. Next, we consider $\bar u$ and $\bar\theta$ satisfying 
\eqref{eq} with initial data $\tilde{\theta}, \tilde{u}$ respectively. Hence, since $\bar{\theta}$ is constant and positive (again, we use Theorem \ref{Twoddolu}), we have
\begin{align}\label{szesc}
\begin{cases}
\bar u_{tt}-\bar u_{xx}=0,\\
0=\mu\bar \theta\bar u_{tx},\\
\bar{u}(a)=\bar{u}(b)=0.
\end{cases}
\end{align}

The second equation and boundary conditions in \eqref{szesc} show that $\bar u_t(t,x)=0$ for all $(t,x)\in[0,\infty)\times[a,b]$. Hence, $\bar{u}_t=0$ in $[a,b]$ and so $\bar{u}$ is constant 
in time for each $x\in[a,b]$, which together with the first equation in \eqref{szesc} states that $\bar{u}_{xx}=0$ and this implies that $\bar{u}=0$.  
Finally, the conservation of the energy (Proposition \ref{energy_cons}) states that
\begin{align*}
\frac12&\iab v_0^2dx+\frac12\iab u_{0,x}^2dx+\iab\theta_0 dx\\
&=\lim_{n\to\infty}\left(\frac12\iab u_t^2(t_n,x)dx+\frac12\iab u_x^2(t_n,x)dx+\iab\theta(t_n,x) dx\right)\\
&=\iab\tilde\theta dx=(b-a)\tilde\theta.
\end{align*}
\end{proof}

\noindent{\bf Data availability} Data sharing not applicable to this article as no datasets were generated or analysed during
the current study.

\noindent{\bf Conflict of interest} The author declares that he has no conflict of interest.

\end{document}